\documentclass[a4paper,english]{article}

\usepackage{lastpage}
\usepackage{amssymb,amsmath}
\usepackage{graphicx}
\usepackage[sf,SF]{subfigure}
\usepackage{caption}
\usepackage{amsthm,euscript}
\usepackage[a4paper, portrait, margin=1in]{geometry}
\usepackage[bookmarks=true,hidelinks]{hyperref}
\usepackage{units}

\newtheorem*{maintheorem}{Main Theorem}
\newtheorem{theorem}{Theorem}
\newtheorem{proposition}[theorem]{Proposition}

\newtheorem{remark}[theorem]{Remark}

\newtheorem{lemma}[theorem]{Lemma}

\numberwithin{equation}{section}
\numberwithin{theorem}{section}
\numberwithin{table}{section}

\newcommand{\Leb}{\EuScript{L}}
\newcommand{\B}{\EuScript{B}}
\newcommand{\R}{\mathbb{R}}

\newcommand{\Z}{\mathbb{Z}}
\newcommand{\Dx}{{\Delta x}}
\newcommand{\Dt}{{\Delta t}}
\renewcommand{\leq}{\leqslant}
\renewcommand{\geq}{\geqslant}
\renewcommand{\phi}{\varphi}
\newcommand{\Lip}{\mathrm{Lip}}

\newcommand{\hf}{{\unitfrac{1}{2}}}
\newcommand{\supp}{\mathrm{supp}}
\newcommand{\sgn}{\mathrm{sgn}}
\renewcommand{\epsilon}{\varepsilon}
\newcommand{\iphf}{{i+\hf}}

\newcommand{\imhf}{{i-\hf}}

\newcommand{\cell}{{\mathcal{C}}}

\allowdisplaybreaks

\title{Convergence rates of the front tracking method for \\ conservation laws in the Wasserstein distances}
\author{Susanne Solem}
\date{}

\begin{document}

\maketitle

\begin{abstract}We prove that front tracking approximations to scalar conservation laws with convex fluxes converge at a rate of $\Dx^2$ in the 1-Wasserstein distance $W_1$. Assuming positive initial data, we also show that the approximations converge at a rate of $\Dx$ in the $\infty$-Wasserstein distance $W_\infty$. Moreover, from a simple interpolation inequality between $W_1$ and $W_\infty$ we obtain convergence rates in all the $p$-Wasserstein distances: $\Dx^{1+\unitfrac{1}{p}}$, $p \in [1,\infty]$.
\end{abstract}

\section{Introduction}
\label{sec:intro}
In this paper we will consider front tracking approximations to the scalar conservation law
\begin{equation}\label{eq:cons_law}
\begin{split}
u_t +f(u)_x = 0,& \qquad x \in \R, \ t >0, \\
u(x,0) = u_0(x),&
\end{split}
\end{equation}
where $f$ is convex, $u_0$ is of compact support and $\Lip^+$ bounded. A function $v$ is said to be $\Lip^+$ bounded if
\begin{equation}\label{eq:OSLC}
\frac{v(x+z)-v(x)}{z} \leq C, \qquad \forall \ x,z \in \R, \quad z \neq 0,
\end{equation}
for some constant $C$. Under these assumptions on $f$ and $u_0$, it is well-established that \eqref{eq:cons_law} admits an unique entropy solution $u$ and that $u(t)$ satisfies \eqref{eq:OSLC} for any $t > 0$, see for example \cite{Oleinik59, Oleinik63, Volpert67, Kruzkov70}. The aim of this paper is to prove the following theorem.
\begin{maintheorem}
Let $f$ be convex and let $u_0$ be of compact support. Then the front tracking approximations of \eqref{eq:cons_law} converge to the unique entropy solution of \eqref{eq:cons_law} at a rate of $\Dx^2$ in the 1-Wasserstein distance. If in addition $u_0 \geq 0$, the approximations converge at a rate of $\Dx$ in the $\infty$-Wasserstein distance, and hence, they converge at a rate of $\Dx^{1+\unitfrac{1}{p}}$ in the $p$-Wasserstein distance $W_p$ for any $p \in [1,\infty]$.
\end{maintheorem}

\noindent See Section \ref{sec:result} for a complete statement of the theorem. This result demonstrates that front tracking approximations (for the class of initial data considered here) converge at a higher rate in every Wasserstein distance than the optimal rate of $\Dx$ in the usual metric $L^1$. Thus it supports the argument that the Wasserstein distances are well-suited to measure the approximation error of solutions to \eqref{eq:cons_law}, which we started in \cite{FjoSo16} (for $W_1$) and continue below. Furthermore, these convergence results gives the front tracking method an advantage, in terms of guaranteed convergence rate, over formally higher-order finite volume approximations to \eqref{eq:cons_law} (for which no convergence rate estimate exists for general initial data in either $L^1$ or $W_p$).

In order to prove the main theorem, we will need (and establish) some properties of solutions to\eqref{eq:cons_law}. Among other things, we will prove stability estimates in $W_1$ and $W_\infty$, show that the support of the solution $u(t)$ is connected if the support of $u_0$ is, and revive a result by Ole\u{\i}nik \cite{Oleinik59}.

\subsection{The Wasserstein distances}
The $p$-Wasserstein distance (or $W_p$-distance), also called the $p$-Monge--Kantorovich distance, is a metric on the set of probability measures with finite $p$th order moment, and for two probability measures $\mu$ and $\nu$ on $\R^d$ it takes the form
\begin{align}\label{eq:the_wp_distance}
W_p(\mu,\nu)= \left( \inf_{\pi \in \Pi(\mu,\nu) } \int_{\R^{2d}} |x-y|^p d \pi(x,y) \right)^{\unitfrac{1}{p}}, \quad p \in [1,\infty),
\end{align}
where the infimum is taken over all measures $\pi$ on $\R^{2d}$ with marginals $\mu$ and $\nu$. See \cite{Vil03} for further details. The $W_\infty$-distance,
\begin{equation}\label{eq:Winfty}
W_\infty(\mu,\nu) = \lim_{p \to \infty} W_p(\mu,\nu),
\end{equation}
is a metric on the space of probability measures with bounded support. 
Although normally only defined for probability measures, the $p$-Wasserstein distance between two Borel measurable functions $u,v \geq 0 $, each of the same finite mass and with finite $p$th order moment for $1\leq p<\infty$,
\begin{align}\label{eq:wpconditions}
\int_{\R^d} (u-v)(x) dx = 0, \quad \int_{\R^d} |x|^p u(x) dx < \infty, \quad \int_{\R^d} |x|^p v(x) dx < \infty,
\end{align}
and of compact support for $p=\infty$, $W_p(u,v):=W_p(u\Leb,v\Leb)$ is well-defined. Here $\Leb$ denotes the Lebesgue measure.

All the $W_p$-distances are suited to measure the difference between (approximate) solutions to \eqref{eq:cons_law}. If $u_0,v_0$ initially fulfil the conditions \eqref{eq:wpconditions}, then the two solutions $u(t), v(t)$ of \eqref{eq:cons_law} (possibly with different flux functions $f,g$ for $u,v$ respectively) will satisfy \eqref{eq:wpconditions} at any later time $t$ due to conservation of mass and finite speed of propagation. Hence, $W_p\bigl(u(t),v(t)\bigr)$ will be well-defined and finite as long as $W_p(u_0,v_0)$ is. To some extent, one can argue that the Wasserstein metrics are natural distances associated to \eqref{eq:cons_law}. Indeed, heuristically the $W_p$ metrics measure the minimum ``cost''  of transporting mass from one measure to another, and transporting quantities (of ``mass'') is exactly what \eqref{eq:cons_law} does.

The 1-Wasserstein distance seems to be particularly suitable in the context of \eqref{eq:cons_law}. To see why, consider the shock and its approximation (stipled) in Figure \ref{fig:wasserstein}(a). The $L^1$-distance, which is commonly used to measure approximation errors of \eqref{eq:cons_law}, measures the area (in grey) between the two solutions. The height is $O(1)$ and the width $O(\Dx)$. Hence, the $L^1$-error between the two solutions is $O(1)\cdot O(\Dx)=O(\Dx)$. The $W_1$-distance can be thought of as measuring the minimal amount of work (mass $\times$ distance) required to move mass from one measure to another. In Figure \ref{fig:wasserstein}(a) this means that $W_1$ measures the work needed to move the surplus of mass to the right of the shock (light grey) to the shortage of mass to the left of the shock (dark grey). The mass (area) to be moved is $O(\Dx)$, and it needs to be moved a distance $O(\Dx)$. It follows that the $W_1$-error is $O(\Dx)\cdot O(\Dx)=O(\Dx^2)$. The difference in the convergence rate between $L^1$ and $W_1$ for shock solutions has already been observed in the case of monotone finite volume scheme approximations. Teng and Zhang \cite{zhang_teng} obtained a convergence rate of $O(\Dx)$ in $L^1$ for solutions consisting of a finite number of decreasing shocks, whereas the rate was improved to $O(\Dx^2)$ in $W_1$ in \cite{FjoSo16}.

We apply the same reasoning to the $W_p$-distance by replacing the distance function $|\cdot|$ with $|\cdot|^p$ and taking the $p$th root to find that the $W_p$-approximation error in Figure \ref{fig:wasserstein}(a) is  $( O(\Dx)\cdot O(\Dx^p))^{\unitfrac{1}{p}}=O(\Dx^{1+\unitfrac{1}{p}})$.

It is not given that there is always a gain in the convergence rate by utilizing one of the Wasserstein distances instead of the $L^1$-distance. Figure~\ref{fig:wasserstein}(b) depicts one such counterexample. Let the $L^1$-error between the solution and its approximation (stipled) be $O(\Dx)$. If the distance between the surplus of mass (light grey) and the shortage of mass (dark grey) is $O(1)$, the error will be $(O(\Dx)\cdot O(1))^{\unitfrac{1}{p}}=O(\Dx^{\unitfrac{1}{p}})$ in $W_p$. Therefore, to obtain a higher rate in the $W_p$-distances, the approximation of the initial data can only redistribute small amounts of mass over small intervals. Furthermore, this redistribution of mass between the approximate and exact solution has to be (close to) preserved at any later time. In this paper we will see that this is the case for the front tracking approximation (which is a first order approximation in $L^1$) and, as a consequence, obtain the $O(\Dx^{1+\unitfrac{1}{p}})$-rate in $W_p$.

\begin{figure}
\centering
\subfigure[Exact and approximate shock solution of \eqref{eq:cons_law}.]{\includegraphics[width=0.4\textwidth]{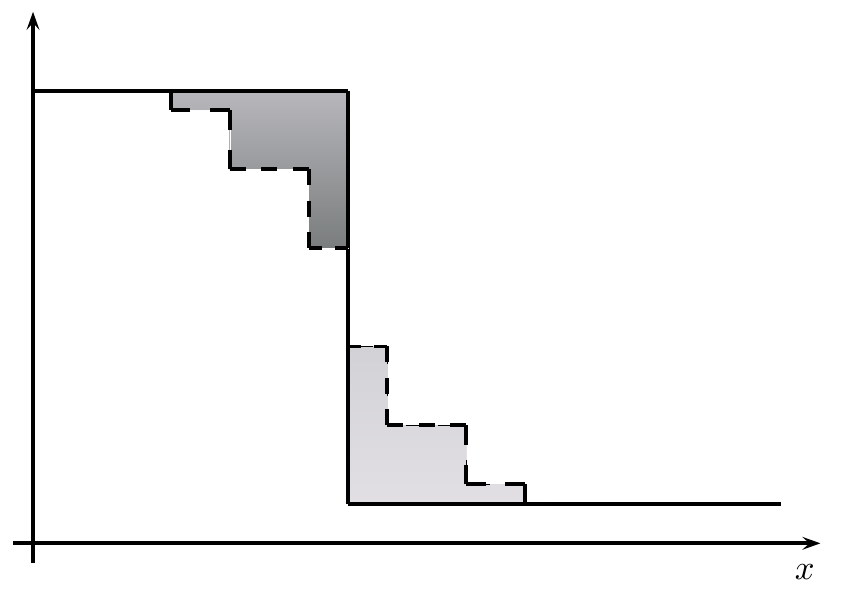}}
\subfigure[Exact and bad (in a $W_1$ sense) approximate solution of \eqref{eq:cons_law}.]{\includegraphics[width=0.4\textwidth]{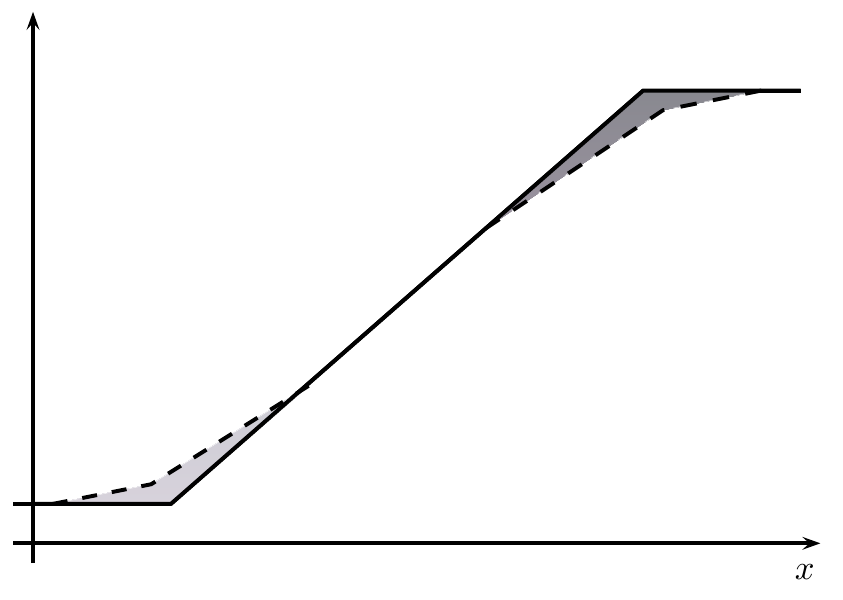}\hspace{4ex}}
\caption{The $W_1$-distance measures the amount of work required to move mass from one place (dark grey) to another (light grey).}
\label{fig:wasserstein}
\end{figure}

Lastly, Carrillo et al. \cite{Caretal2006} have shown that the $W_\infty$-distance is contractive with respect to initial data for solutions of \eqref{eq:cons_law} --- a property that will be exploited in this paper.

\subsection{Front tracking, finite volume methods and convergence rates}

The front tracking method was first proposed by Dafermos \cite{Daf72}. Later, Holden et al. \cite{HHK88} rediscovered it, extended it to non-convex fluxes and showed that it is a viable numerical method in one dimension. The main strength of the one-dimensional front tracking method is that the approximation is itself an entropy solution to a conservation law. We will make use of this strength in this paper by first proving general stability results of \eqref{eq:cons_law} in both $W_1$ and $W_\infty$ and then applying them to the front tracking method in order to obtain the respective $\Dx^2$ and $\Dx$ rates.

Up to this point the $W_1$-distance is the only one among the Wasserstein distances that has been applied in order to study convergence rates of approximations to \eqref{eq:cons_law}. Tadmor et al.\ \cite{Tad91,nessy_conv92,nessy_conv} extensively examined it in the context of conservation laws, but under the different name of the \emph{Lip'-norm}. They showed (among other things) that a large class of monotone finite difference methods converge at a rate of $\Dx$ in the Lip'-norm for initial data $u_0$ of compact support satisfying \eqref{eq:OSLC}. By applying their technique to the front tracking approximation, one obtains the rate $\Dx$ in $W_1$.

Due to the structure of the front tracking method, the (provable) convergence rate of this approximation is usually higher than the one for monotone methods. This can be observed in the above for the rates in $W_1$ ($\Dx^2$ for front tracking and $\Dx$ for monotone schemes), and can also be noticed for the rates in $L^1$. By applying a well-known stability result in the $L^1$ norm, first proved by Lucier \cite{Luc86}, one attains the (optimal) convergence rate $\Dx$ in $L^1$ of the front tracking approximation. However, the most generic result on convergence rates of monotone methods is the $O(\Dx^\hf)$ rate in $L^1$ due to Kuznetsov \cite{Kuz76}. A counterexample due to \c Sabac shows that the $\Dx^\hf$ rate for monotone methods is sharp and cannot be improved without further assumptions on the initial data \cite{Sab97}. But, even though it is not proved, it should be noted that numerical evidence indicates that the convergence rate is close to $\Dx$ for monotone schemes as well in the case of more ``natural'' initial data. The rate of $O(\Dx)$ in $L^1$ for a finite number of travelling shocks in \cite{zhang_teng} endorses these observations.

The first proof of a second-order convergence rate of any numerical method to \eqref{eq:cons_law} was provided in $L^1$ by Lucier for a specific piecewise linear extension of the front tracking method \cite{Luc86}. In this paper we prove the same rate in $W_1$ without modifying the original method.

In \cite{KR02} Karlsen and Risebro demonstrate the equivalence between entropy solutions of conservation laws and viscosity solutions of the Hamilton--Jacobi equations by utilizing the front tracking method. As a by-product they discover the rate $\Dx^2$ in the $L^\infty$ distance between the primitives of the front tracking approximation and the entropy solution. This result is closely related to the rate of $\Dx^{2}$ in $W_1$ that we obtain in this paper, see Remark~\ref{rem:hjft}. Hong \cite{hong96} proved a stability result in $L^\infty$ for the Hamilton--Jacobi equations, from which one can also deduce a $\Dx^2$ rate.

Apart from the second-order rate results for front tracking type methods in \cite{Luc86, KR02}, the only other proof of a second-order rate in any norm of any numerical method for \eqref{eq:cons_law} is, to the authors knowledge, the $\Dx^2$ rate in the 1-Wasserstein distance in \cite{FjoSo16}. 

Except for a piecewise constant projection of the initial data, the one-dimensional front tracking method is grid independent. A simple way to extend the method to the multi-dimensional case, is by dimensional splitting, see \cite{HR93}. The multi-dimensional extension is no longer the exact solution of a conservation law, and the accuracy of the method depends on the temporal grid size $\Dt$. Thus, the method is no longer grid independent, and the convergence rate of the method might decrease. The two-dimensional method is proven to converge at a rate of $O(\Dt^\hf + \Dx^\hf)$ in $L^1$ \cite{Karlsen94, Teng94}, which is the same rate as the one Kuznetsov proved for multi-dimensional monotone schemes. Whether the rate improves in $W_1$, is not known. All studies of convergence rates of approximations to \eqref{eq:cons_law} in $W_1$, including the one in this paper, rely on a one dimensional interpretation of $W_1$ which does not extend to multiple dimensions.

The $W_1$ rate in this paper shows that the front tracking approximation to \eqref{eq:cons_law} can be considered a second-order method when applying a suitable metric (although this might be restricted to one dimension). Indeed, one can observe numerically that the front tracking approximation converges at the same rate $\Dx^2$ as a second-order finite volume method in $W_1$. Furthermore, the $\Dx$ rate in $W_\infty$ conveys that displaced mass in the front tracking approximation compared to the exact solution of \eqref{eq:cons_law} is moved at most a distance $\Dx$.

Next follows an outline of this paper. In Section \ref{sec:result} we provide a short basis for the upcoming results and associated proofs before stating the main theorem. Section \ref{sec:w1} contains stability estimates in the $W_1$-distance, which provide the convergence rate in $W_1$. Section \ref{sec:winfty} is devoted to the proof of the rate in $W_\infty$. Lastly, Section \ref{sec:conclude} contains remarks on possible extensions of the main theorem.

\section{Front tracking, Wasserstein metrics and main theorem}\label{sec:result}
\subsection{The front tracking method}
\label{sec:fronttrack}
These are the main ingredients in the front tracking method. Approximate the initial data $u_0$ in \eqref{eq:cons_law} by a piecewise constant function $u_0^\Dx$ and the flux $f$ by a piecewise linear function $f^\delta$. Then solve the resulting conservation law
\begin{equation}\label{eq:ft}
u_t + f^\delta (u)_x=0, \qquad u(x,0)=u_0^\Dx(x),
\end{equation}
exactly. As $u_0^\Dx$ is piecewise constant, the initial problem will be to solve a series of independent Riemann problems, each of them having a wave-front traveling with constant speed, due to $f$ being piecewise linear, as a solution. Whenever two fronts meet, we restart the procedure by solving \eqref{eq:ft} with initial data $u(x,0)=u^{\delta,\Dx}(x,t^c)$, where $t^c$ is a interaction time. In this way we can find $u^{\delta,\Dx}(x,t)$ for all times. The resulting solution $u^{\delta,\Dx}$ is the unique entropy solution to \eqref{eq:ft}.

As the Wasserstein distances require that the functions to be compared have equal mass, we approximate the initial data as
\begin{equation}\label{eq:ftinitial}
u^\Dx_0(x) = u_i := \frac{1}{\Dx}\int_{\cell_i} u_0(y) dy, \qquad x \in \cell_i := [x_\imhf,x_\iphf), \quad i \in \Z,
\end{equation}
where $x_\iphf=(\iphf)\Dx$, $\Dx>0$, to ensure that $\int_\R (u_0-u^\Dx_0) \ dx = 0$. (In general one can use piecewise constant approximations that are not necessarily tied to the grid on $\R$ or preserves the mass.) The front tracking flux $f^\delta$ is a piecewise linear approximation to $f$ of the following form
\begin{equation}\label{eq:ftflux}
f^\delta (u) = f(j\delta ) + (u -j\delta) \frac{f\bigl((j+1)\delta\bigr) -f\bigl(j\delta\bigr)}{\delta}, \qquad u \in (j\delta,(j+1)\delta],
\end{equation}
for $j \in \Z \cap [-(M+1)/\delta, M/\delta]$, where $M=\|u_0\|_{L^\infty(\R)}$ and $\delta = O(\Dx)$. See \cite{Daf72, HHK88, HR15, Luc86} for more details on the method. 

\subsection{The Wasserstein distances in one dimension}\label{sec:wassonedim}
Without loss of generality, let $\int_\R |u(x)| \ dx = 1$ in \eqref{eq:cons_law} from this point on. We define the two spaces
\begin{align*}
\B_p &:= \left\{ u \in L^\infty(\R):  u \geq 0, \int_\R u(x) \ dx = 1, \int_\R |x|^p u(x) dx < \infty  \right\}, \quad p \in [1,\infty),\\
\B &:= \left\{u \in L^\infty(\R) : u \geq 0, \ \supp(u) \ \textrm{compact}, \int_\R u(x) \ dx = 1  \right\},
\end{align*}
for ease of notation. 

In one dimension, the $p$-Wasserstein distance \eqref{eq:the_wp_distance} between $u$ and $v$ both in $\B_p$, has a simple interpretation as the $L^p$-distance between the pseudo-inverses of the distribution functions \cite{CarTosc05, Vil03},
\begin{align}\label{eq:antiderivatives}
U(x) = \int_{-\infty}^x u(y) \ dy, \quad V(x) = \int_{-\infty}^x v(y) \ dy.
\end{align}
The pseudo-inverses $U^{-1}: [0,1] \to \R$, $V^{-1}: [0,1] \to \R$ are defined as
\begin{align*}
U^{-1}(\xi) = \inf \{x: \ U(x) > \xi \}, \quad V^{-1}(x) = \inf \{x: \ V(x) > \xi \}.
\end{align*}
Then the $p$-Wasserstein distance, $p \in [1, \infty)$ is 
\begin{equation}\label{eq:wp_onedim}
W_p(u,v) = \| U^{-1} - V^{-1}\|_{L_p([0,1])}.
\end{equation}
When $u,v \in \B$, we can interpret the $W_\infty$-distance in the same way using \eqref{eq:Winfty},
\begin{align} \label{eq:winfty_onedim}
W_\infty (u,v) = \lim_{p \to \infty} \| U^{-1} - V^{-1}\|_{L_p([0,1])} = \| U^{-1} - V^{-1}\|_{L_\infty([0,1])}.
\end{align}
In particular, the $W_1$-distance takes the very simple form
\begin{equation}\label{eq:wassalternative}
W_1(u,v) = \int_{\R}\left| U-V \right| dx,
\end{equation}
which can be found by using Fubini's theorem with \eqref{eq:wp_onedim}. Notice that for the alternative form \eqref{eq:wassalternative} of $W_1$ to be well-defined, $u$ and $v$ only need to satisfy 
\begin{align}\label{eq:wassercondition}
\int_\R (u-v)(x) dx = 0 , \qquad \int_\R |x||u-v|(x) dx < \infty.
\end{align}

\subsection{Connection to the Hamilton--Jacobi equation}
There is a well-known equivalence between the viscosity solution of the Hamilton--Jacobi equation  
\begin{equation}\label{eq:HJ}
U_t +f(U_x) = 0, \qquad  \int^x u_0 \ dx := U_0 \in BUC(\R)
\end{equation}
and the entropy solution of \eqref{eq:cons_law} with $u_0 \in BV(\R)$ through the relation
\begin{equation}\label{eq:equivalence}
\partial_x U = u, \qquad U = \int^x u \ dx,
\end{equation}
see \cite{KR02} and references therein. If $U_0$ is Lipschitz continuous, bounded and $f$ is convex and superlinear,
\begin{align*}
\lim_{|u| \to \infty} \frac{f(u)}{|u|} = \infty,
\end{align*}
then the unique viscosity solutions of \eqref{eq:HJ} can be found by the Hopf--Lax formula
\begin{equation}\label{eq:laxhopf}
U(x,t) = \min_{y \in \R}\left\{tf^*\left(\frac{x-y}{t} \right) + U_0(y) \right\},
\end{equation}
where $f^*$ is the Legendre transform of $f$,
\begin{align}\label{eq:legendre}
f^*(p) = \sup_{u \in \R} \left\{pu-f(u)\right\} 
\end{align}
(see for example Evans \cite[Ch. 3.3, Ch. 10.3.4]{Evans}). Then according to \eqref{eq:equivalence}, we get the entropy solution of \eqref{eq:cons_law} by differentiating \eqref{eq:laxhopf} with respect to $x$.

\subsection{Main theorem}
The main theorem relies on the following interpolation result. 
\begin{lemma}
If $u,v \in \B$, then 
\begin{equation}\label{eq:inter_ineq}
W_p(u,v) \leq W_1(u,v)^{\unitfrac{1}{p}}W_\infty(u,v)^{1-\unitfrac{1}{p}},
\end{equation}
for $1 < p < \infty$.
\end{lemma} 
\begin{proof}
By H\"older's inequality and the representation \eqref{eq:wp_onedim},
\begin{align*}
W_p(u,v)^p = \| |F^{-1}-G^{-1}|^p \|_{L^1([0,1])} \leq& \ \| (F^{-1}-G^{-1}) \|_{L^1([0,1])}\| |F^{-1}-G^{-1}|^{p-1} \|_{L^\infty([0,1])} \\
= & \ \| F^{-1}-G^{-1} \|_{L^1([0,1])} \| F^{-1}-G^{-1}\|_{L^\infty([0,1])}^{p-1} \\
= & \ W_1(u,v)W_\infty(u,v)^{p-1}.
\end{align*}
By taking the $p$th root, we get the interpolation inequality \eqref{eq:inter_ineq}.
\end{proof}

\begin{theorem}\label{thrm:main}
Let $u$ be the entropy solution of \eqref{eq:cons_law} where $f$ is twice continuously differentiable and convex, and $u_0 \in BV(\R)$ be $\Lip^+$ bounded and of compact support. Then the front tracking approximation $u^{\delta,\Dx}$ of $u$ satisfies
\begin{align}
W_1\big(u(t),u^{\delta,\Dx}(t)\big) &\leq C \Dx^2 \nonumber
\intertext{for all $t \in [0,T)$ for any $T>0$. If in addition $u_0 \in \B $, then}
W_p\big(u(t),u^{\delta,\Dx}(t)\big) &\leq C \Dx^{1+\unitfrac{1}{p}}, \quad p \in [1,\infty], \nonumber
\end{align}
where $C:=C_{T,f,\supp(u_0)}$.
\end{theorem}
\begin{proof}
This follows directly from the second order rate in $W_1$ and the first order rate in $W_\infty$ to be proved in Theorem \ref{thrm:wass1} and Theorem \ref{thrm:wass_infty} respectively, and from the interpolation inequality \eqref{eq:inter_ineq}.
\end{proof}

\section{The convergence rate in $W_1$}\label{sec:w1}
We begin by providing two stability estimates in the 1-Wasserstein distance that will yield the second-order convergence rate.
\begin{proposition}\label{lem:w1stability}
Assume that $f$ is continuously differentiable such that $f'$ is locally Lipschitz. Let $f$ and $g$ both be convex, and let $u_0, v_0 \in BV (\R) $ satisfy \eqref{eq:wassercondition} and \eqref{eq:OSLC}. Then the entropy solutions $u$ and $v$ of
\begin{equation}\label{eq:2cons_laws}
\begin{gathered}
u_t + f(u)_x = 0, \qquad u(x,0)=u_0(x), \\
v_t + g(v)_x = 0, \qquad v(x,0)=v_0(x),
\end{gathered}
\end{equation}
satisfy
\begin{align}\label{eq:w1stability}
W_1\big( u(t),v(t) \big) \leq C(t) \left[W_1\big( u_0, v_0 \big) + \int_0^t \|(f-g)(v(s))\|_{L^1(\R)} ds \right].
\end{align}
\end{proposition}
\begin{proof} As previously mentioned, since $u_0-v_0$ satisfies \eqref{eq:wassercondition}, $(u-v)(t)$ will also fulfil the same conditions by conservation of mass and finite speed of propagation. Hence $W_1(u(t),v(t))$ is well-defined and finite.

We start by differentiating \eqref{eq:wassalternative} with respect to $t$ and using \eqref{eq:antiderivatives},
\begin{align}
\frac{d}{dt} W_1(u,v) =& \int_\R \partial_t \left|(U-V)(x) \right| dx \nonumber\\
=& \int_\R \sgn \left(U-V\right)(x) \partial_t (U-V) dx = -\int_\R \sgn \left(U-V\right)(x) \left(f(\partial_x U) -g(\partial_x V) \right) dx \nonumber \\
= & - \int_\R a(u,v) \partial_x \left|U-V \right| dx + \int_\R \sgn \left(U-V\right)(x) \left( (g-f)(v(x)) \right) dx, \label{eq:inter_calc} 
\end{align}
where 
\begin{align*}
a(u,v) = \int_0^1 f'(\alpha u +(1-\alpha)v) d\alpha.
\end{align*}
Note that $U-V$ is differentiable in $t$ due to the Lipschitz continuity in time of $u$ and $v$ with respect to the $L^1$ norm. From an integration by parts we find that the first term in \eqref{eq:inter_calc} is
\begin{align*}
 - \int_\R a(u,v) \partial_x \left|U-V \right| dx = \int_\R  \left|U-V \right| D_x a(u,v) dx,
\end{align*}
where we give meaning to $D_x a (u,v)$ as a distributional derivative ($|U-V|$ is Lipschitz, but $a(u,v)$ might contain decreasing jumps). This leads to the following upper bound on the time derivative of the $W_1$-distance,
\begin{align*}
\partial_t W_1(u(t),v(t)) \leq \sup_{x \in \R} \left(\partial_x a(u,v)(t)\right)^+  W_1(u(t),v(t)) + \|(g-f)(v(t)) \|_{L^1(\R)}.
\end{align*}
Note that $(D_x a(u,v))^+=(\partial_x a(u,v))^+$ as $u$ and $v$ are $\Lip^+$ bounded. By Gr\"onwall's inequality we deduce that \eqref{eq:w1stability} holds with
\begin{align}\label{eq:C(t)}
C(t) := \exp \left( \| f'\|_\Lip C t \right) \geq \exp \left( \int_0^t \sup_{x \in \R} \left(\partial_x a(u,v)(s) \right)^+ ds \right),
\end{align}
as $a(u,v)$ is increasing in both $u$ and $v$ and both $u(t)$ and $v(t)$ satisfy \eqref{eq:OSLC}. The constant $C$ is the constant in \eqref{eq:OSLC}.
\end{proof}
A similar stability result was established by Nessyahu and Tadmor \cite{nessy_conv92} by studying the dual equation of $u-v$, i.e. the backward in time equation for the dual $\phi$ in the Kantorovich--Rubinstein formulation of the $W_1$-distance, see \cite[Thm. 1.14]{Vil03} for the definition.
\begin{remark}
If $u_0, v_0$ are non-increasing and $f=g$ in Proposition \ref{lem:w1stability}, then \eqref{eq:w1stability} reduces to the contraction estimate
\begin{align*}
W_1\big( u(t),v(t) \big) \leq W_1\big( u_0,v_0 \big).
\end{align*}
\end{remark}

In \eqref{eq:w1stability} it is necessary that $u(t)$ and $v(t)$ satisfy \eqref{eq:OSLC}. As front tracking approximations consist of piecewise constants, they will in general not fulfil this condition. In order to overcome this obstacle without risking to sacrifice the second-order convergence rate, we will utilize an old result by Ole\u{\i}nik \cite[Theorem 2]{Oleinik59}:

\begin{theorem}[Ole\u{\i}nik \cite{Oleinik59}]\label{thrm:oleinik}
Let $f$ be twice continuously differentiable (and not necessarily convex). Assume that $u$ and $v$ are two piecewise smooth solutions of \eqref{eq:cons_law} which satisfy Ole\u{\i}niks condition E. Then if
\begin{align*}
\left| \int_{y_1}^{y_2} u_0(y)-v_0(y) \ dy \right| \leq c \quad \text{for all $y_1,y_2 \in [a,b]$, then} 
\quad \left| \int_{x_1}^{x_2} u(y,t)-v(y,t) \ dy \right| \leq c
\end{align*} for all $x_1,x_2$ in the smaller interval $\bigl[a+Qt,b-Qt\bigr]$, $Q=\|f\|_\Lip$. 
\end{theorem}
We will extend the above result to all $y_1,y_2,x_1,x_2 \in \R$ and to $f$ only locally Lipschitz. To ensure that the piecewise smoothness assumption in Theorem \ref{thrm:oleinik} is satisfied, we will assume that $f$ is convex.

\begin{lemma}\label{lem:lessthaneps}
Let $u_0, v_0 \in BV(\R)$. Consider the respective entropy solutions $u$ and $v$ of \eqref{eq:cons_law} where $f$ is assumed to be convex. If there exists $c>0$ s.t.
\begin{align}\label{eq:init_lessthaneps}
\left| \int_{y_1}^{y_2} u_0(y)-v_0(y) \ dy \right| & \leq c \\
\intertext{for all pairs $y_1,y_2 \in \R$, then}
\left| \int_{x_1}^{x_2} u(x,t)-v(x,t) \ dx \right| & \leq c \nonumber
\end{align}
for all $x_1,x_2 \in \R$ for any finite time $t>0$.
\end{lemma}
\begin{proof}
We start by approximating the initial data $u_0$ and $v_0$ by smooth functions of compact support, $\tilde{u}_0$ and $\tilde{v}_0$, such that 
\begin{align}\label{eq:initial_approx}
\|\tilde{u}_0-u_0\|_{L^1(\Omega)} <\epsilon \quad \textrm{and} \quad  \|\tilde{v}_0-v_0\|_{L^1{(\Omega})} < \epsilon
\end{align}
on a finite interval $\Omega$ (to be determined). Then if $f$ is strictly convex and smooth, $\tilde{u}(t)$ and $\tilde{v}(t)$ will be piecewise smooth, see \cite{Daf77, TT93} for example, and Ole\u{\i}niks condition E in \cite{Oleinik59} will be satisfied. As both the approximate initial data are of compact support, the lemma then follows directly for $\tilde{u}(t),\tilde{v}(t)$ from Theorem \ref{thrm:oleinik} for strictly convex and smooth $f$.

As $u_0 \in BV(\R)$, $\|u_0\|_{L^\infty(\R)}\leq M$ for some constant $M$. We extend the result to Lipschitz continuous $f$ by approximating $f$ by a sequence of twice continuously differentiable strictly convex flux functions $f^\epsilon$ such that $\|f-f^\epsilon\|_{L^\infty(\R)} < \epsilon^3$. Such a function can be found by mollifying $f$ and then adding $a_\epsilon u ^2$ for a suitably small $a_\epsilon$ such that $a_\epsilon \to 0$ when $\epsilon \to 0$. Notice that we can choose $f^\epsilon$ such that $\|f^\epsilon\|_\Lip \leq \|f\|_\Lip + a_\epsilon M $ on $[-M,M]$. Let $\tilde{u}^\epsilon$ be the entropy solution of \eqref{eq:cons_law} with $f^\epsilon$ as a flux function and $\tilde{u}_0$ as initial data. From a $L^1$-stability estimate by Bouchut and Perthame \cite[Thm. 3.1~\textit{(iii)}]{BoPert98}, we find that
\begin{align*}
\int_{\Omega(t)} |\tilde{u}(x,t) -\tilde{u}^\epsilon (x,t) | \ dx \leq & \ K \Bigl[\left(|\Omega(t)| + Qt \right)TV(\tilde{u}_0)t\|f-f^\epsilon-(f-f^\epsilon)(0)\|_{L^\infty}\Bigr]^\hf\\
\leq & \ K \Bigl[\left(|\supp (\tilde{u}_0)| + 3Qt \right)TV(u_0)t\|f-f^\epsilon-(f-f^\epsilon)(0)\|_{L^\infty}\Bigr]^\hf \\
< & \ K \Bigl[ \left( |\supp (\tilde{u}_0)| + 3Qt \right) TV(u_0)t \Bigr]^\hf \epsilon^{\unitfrac{3}{2}}
\end{align*}
where $K$ is an absolute constant and $\Omega(t)$ is the maximal support of $\tilde{u}(t)$ and $\tilde{u}^\epsilon(t)$, and $Q = \|f\|_\Lip + a_\epsilon M$ . The same can be done for $\tilde{v}$. We extend the result to $u_0,v_0 \in BV(\R)$ by choosing the support of $\tilde{u}_0$ (and of $\tilde{v}_0$) to be $\big[-\frac{1}{2\epsilon}, \frac{1}{2\epsilon}\big]$ in order to get 
\begin{align}\label{eq:epsilonestimate}
\int_{\Omega(t)} |\tilde{u}(x,t) -\tilde{u}^\epsilon (x,t) | \ dx < C(t) \epsilon,
\end{align}
where $C(t):= K \left[1 + (3Qt\epsilon)^\hf \right](TV(u_0)t)^\hf$, from the estimate above. By choosing the (smooth) approximations to have the above support, the interval for which \eqref{eq:initial_approx} holds has to be slightly smaller. We can choose it to be $\Omega = \big[-\frac{1}{2\epsilon} +\epsilon, \frac{1}{2\epsilon}-\epsilon \big]$. 

Assume that \eqref{eq:init_lessthaneps} holds for $u_0,v_0 \in BV(\R)$. Then, by the triangle inequality, \eqref{eq:epsilonestimate} and the $L^1$ contraction property of \eqref{eq:cons_law},
\begin{align*}
\left| \int_{x_1}^{x_2} u(x,t)-v(x,t) \ dx \right| \leq& \left| \int_{x_1}^{x_2} \tilde{u}^\epsilon(x,t)-\tilde{v}^\epsilon(x,t) \ dx \right| \\
& + \int_{x_1}^{x_2} \left| \tilde{u}(x,t)-\tilde{u}^\epsilon(x,t)\right| dx  + \int_{x_1}^{x_2} \left| \tilde{v}(x,t)-\tilde{v}^\epsilon(x,t)\right| dx \\
& + \int_{x_1}^{x_2} \left| \tilde{u}(x,t)-u(x,t) \right| dx  + \int_{x_1}^{x_2} \left| \tilde{v}(x,t)-v(x,t)\right| dx \\
< & \ c + 2C(t)\epsilon + \|\tilde{u}(t)-u(t)\|_{L^1([x_1,x_2])} + \|\tilde{v}(t)-v(t)\|_{L^1([x_1,x_2])} \\
\leq & \ c + 2C(t)\epsilon + \|\tilde{u}_0-u_0\|_{L^1([x_1-Qt,x_2+Qt])} + \|\tilde{v}_0-v_0\|_{L^1([x_1-Qt,x_2+Qt])} \\
\leq & \ c + 2C(t)\epsilon +  2\epsilon ,
\end{align*}
for any $x_1, x_2 \in \big[-\frac{1}{2\epsilon}+\epsilon+Qt, \frac{1}{2\epsilon} -\epsilon -Qt \big]$, as the theorem holds for $f^\epsilon$ twice continuously differentiable. Letting $\epsilon \to 0$ now yields the result.
\end{proof}
\begin{remark}
As the main result in this paper relies on $f$ being convex, we simply assumed convexity in Lemma \ref{lem:lessthaneps} in order to obtain the piecewise smoothness needed to apply Theorem \ref{thrm:oleinik}. Jennings has shown that piecewise smooth solutions of \eqref{eq:cons_law} do exist for a certain class of $u_0$ and non-convex $f$, see \cite{Jen79}. By an appropriate approximation of $u_0$ and $f$ by functions in this class, Lemma \ref{lem:lessthaneps} should be extendible to non-convex $f$.
\end{remark}

We are now ready to prove that the convergence rate of front tracking approximations is $O(\Dx^2)$ when measured in $W_1$.

\begin{theorem}\label{thrm:wass1}
Assume that $u_0 \in BV(\R)$ is of compact support and satisfies \eqref{eq:OSLC}. Let $f$ be twice continuously differentiable and convex. Furthermore, let $u^{\delta, \Dx}$ be the front tracking solution of \eqref{eq:ft} with initial data \eqref{eq:ftinitial} and flux \eqref{eq:ftflux} such that $\delta = O(\Dx)$. Then
\begin{equation}\label{eq:w1rate}
W_1\left(u(t),u^{\delta,\Dx}(t)\right) \leq \tilde{C}(t) \Dx^2,
\end{equation}
where $u$ is the entropy solution of \eqref{eq:cons_law}. The constant $\tilde{C}(t)$ is defined in \eqref{eq:tildec}.
\end{theorem}
\begin{proof} First observe that with $u^\Dx_0$ as in \eqref{eq:ftinitial}, $ u^\Dx_0-u_0$ satisfies \eqref{eq:wassercondition}, and the $W_1$-distance is well-defined. Also, by a simple calculation using \eqref{eq:wassalternative},
\begin{align}\label{eq:initestimate}
W_1\big(u_0, u_0^{\Dx} \big) \leq \Dx^2 TV(u_0).
\end{align}
In order to use Proposition \ref{lem:w1stability}, we fix an intermediate solution $u^{\delta,\sigma}$ by solving \eqref{eq:ft} replacing $u_0^\Dx$ in \eqref{eq:ftinitial} with the slightly regularized initial data
\begin{align*}
u_0^{\sigma}(x) = 
\begin{cases}
u_i + \frac{u_{i+1}-u_i}{\Dx}(x-i\Dx), \quad x \in [i\Dx,(i+1)\Dx], &\quad \textrm{when} \ u_i < u_{i+1}, \\
u_0^\Dx(x), &\quad \textrm{otherwise}.
\end{cases}
\end{align*}
As $u_0$ satisfies \eqref{eq:OSLC}, it is not hard to see that $u_0^{\sigma}$ also will.
By the triangle inequality,
\begin{align*}
W_1\left(u(t),u^{\delta,\Dx}(t)\right) \leq W_1\left(u(t),u^{\delta,\sigma}(t)\right) + W_1\left(u^{\delta,\sigma}(t),u^{\delta,\Dx}(t)\right) = I + II.
\end{align*}
Applying Proposition \ref{lem:w1stability} to $I$, we get
\begin{align*}
I=W_1\left(u(t),u^{\delta,\sigma}(t)\right) \leq C(t) \left[W_1\big( u_0, u_0^{\sigma} \big) + \int_0^t \|(f-f^\delta)(u^{\delta,\sigma}(s))\|_{L^1(\R)} ds \right],
\end{align*}
where $W_1\big(u_0, u_0^{\sigma} \big) \leq \Dx^2 TV(u_0)$ follows from \eqref{eq:initestimate}. Furthermore,
\begin{align*}
 \int_0^t \|(f-f^\delta)(u^{\delta,\sigma}(s))\|_{L^1(\R)} ds \leq & \ t \max_{s \in [0,t]} \left|\supp \big(u^{\delta,\sigma}(s)\big)\right|\sup_{u \in [-M,M]} \left|(f-f^\delta)(u) \right| \\
\leq & \ \left[ t \max_{s \in [0,t]} \left|\supp \big(u^{\delta,\sigma}(s)\big)\right| \sup_{u \in [-M-\delta,M+\delta]} f''(u)\right] \delta^2,
\end{align*}
where the first inequality follows from $(f-f^\delta)(0)=0$ and the compact support of $u^{\delta,\sigma}$ and the second inequality by \eqref{eq:ftflux} and a Taylor expansion of $f$ around $j\delta$ where $u \in [j\delta,(j+1)\delta]$. The number $M$ is the constant such that $|u_0| \leq M$ (which is finite since $u_0 \in BV(\R)$). It follows from \eqref{eq:ftinitial} that also $|u_0^\sigma|,|u_0^\Dx| \leq M$. Hence, $|u^{\delta,\Dx}(t)|,|u^{\delta,\sigma}(t)| \leq M$. Thus
\begin{align}
\max_{s \in [0,t]} \left| \supp \big(u^{\delta,\sigma}(s)\big) \right| & \leq |\supp(u_0)|+2\Dx + \| f^\delta \|_{\Lip([-M,M])}t \nonumber \\
& \leq |\supp(u_0)|+2\Dx + \| f \|_{\Lip([-M-\delta,M+\delta])}t =: K(t) \label{eq:K(t)},
\end{align}
and similarly for $u^{\delta,\Dx}$.

As $u^{\delta,\sigma}(t)$ and $u^{\delta,\Dx}(t)$ are of compact support and $\int_\R u^{\delta,\sigma}(t)-u^{\delta,\Dx}(t) \ dx = 0$, we estimate $II$ as follows,
\begin{align*}
II=W_1\left(u^{\delta,\sigma}(t),u^{\delta,\Dx}(t)\right) & = \int_\R \left|\int_{-\infty}^x \left(u^{\delta,\sigma}-u^{\delta,\Dx}\right)(y,t) dy \right| dx \\
& \leq   \max \left\{|\supp (u^{\delta,\sigma})|, |\supp (u^{\delta,\Dx})| \right\} \sup_x \left|\int_{\gamma}^x \left(u^{\delta,\sigma}-u^{\delta,\Dx}\right)(y,t) dy \right| \\
& \leq   K(t) \sup_x \left|\int_{\gamma}^x \left(u^{\delta,\sigma}-u^{\delta,\Dx}\right)(y,t) dy \right|,
\end{align*}
where $\gamma$ is the smallest value in the support of $u^{\delta,\sigma}(t)-u^{\delta,\Dx}(t)$. 
Also, for any $x_1,x_2 \in \R$,
\begin{align*}
\left|\int_{x_1}^{x_2} \left(u_0^{\sigma}-u_0^{\Dx}\right)(y,t) dy \right| \leq \frac{1}{8} C \Dx^2, 
\end{align*}
where $C$ comes from \eqref{eq:OSLC}. Thus we can apply Lemma \ref{lem:lessthaneps} to conclude that 
\begin{align*}
\sup_x \left|\int_{\gamma}^x \left(u^{\delta,\sigma}-u^{\delta,\Dx}\right)(y,t) dy \right| \leq \frac{1}{8} C \Dx^2.
\end{align*}
Combining the two estimates $I$ and $II$ gives \eqref{eq:w1rate} with
\begin{align}\label{eq:tildec}
\tilde{C}(t) = \left[ C(t) \left(TV(u_0) + t \lambda^2 K(t) \sup_{u \in [-M,M]} f''(u) \right) + \frac{K(t)}{8} C \right],
\end{align}
where $\lambda = \delta/\Dx$, $C(t)$ is defined in \eqref{eq:C(t)} and $K(t)$ in \eqref{eq:K(t)}.
\end{proof}

\begin{remark}[A different approach to the $\Dx^2$ rate]
\label{rem:hjft}
As mentioned in the introduction, a rate of $\Dx^2$ in the $W_1$-distance can be deduced from a result by Karlsen and Risebro \cite[Remark 2.2]{KR02} for $f \in C^2$ (not necessarily convex) and $u_0 \in C_c^1$. In their paper the focus is on proving the equivalence between entropy solutions of the conservation law \eqref{eq:cons_law} and viscosity solutions of the corresponding Hamilton--Jacobi equation \eqref{eq:HJ} through the relation \eqref{eq:equivalence}. This is done by translating the front tracking method to a method for \eqref{eq:HJ}. As a bonus, the authors find that front tracking approximations to \eqref{eq:HJ} converge at a rate of $\Dx^2$ in the $L^\infty$-norm when $u_0 \in C_c^1$. From \eqref{eq:wassalternative}, it is not hard to see that this translates into a rate of $O(\Dx^2)$ in $W_1$ for \eqref{eq:cons_law},
\begin{equation}\label{eq:2order_HJ}
W_1(u,u^{\delta,\Dx}) \leq C |\Omega(t)| \|U-U^{\delta,\Dx} \|_{L^\infty} =C | \Omega (t) | \Dx^2, 
\end{equation}
where $|\Omega(t)|=|\supp(u(t))\cup \supp(u^{\delta,\Dx}(t))| \leq |\supp (u_0)| + 2\Dx + 2\|f\|_{\Lip([-M-\delta,M+\delta])} t$, and $C$ depends on $\|f''\|_{L^\infty}$, $\|u_0''\|_{L^\infty}$ and the time $t$.

Although the convergence rate \eqref{eq:2order_HJ} is the same as the one we prove in this paper in $W_1$, the approach and the assumptions made differ. The results in this paper rely directly on inequalities involving the $W_1$-metric and does not go via front tracking approximations to solutions of \eqref{eq:HJ}. With this approach we can prove a $\Dx^2$ rate in $W_1$ also in the case of $u_0$ with decreasing jump discontinuities, whereas with the approach in \cite{KR02} the rate will reduce to $\Dx$ for such $u_0$. The drawback is that we have to assume convexity of $f$ which is not required in \cite{KR02}.
\end{remark}

\section{The convergence rate in $W_\infty$}\label{sec:winfty}
In order to prove the $\Dx$ rate in $W_\infty$, we require stability estimates of solutions to \eqref{eq:cons_law} with respect to both the initial data and the flux functions. To obtain these estimates, we will extend the $W_\infty$-contractivity with respect to initial data proved by Carrillo et al. \cite{Caretal2006} to cover the case of the front tracking equation \eqref{eq:ft}. Furthermore, inspired by the proof of the $W_\infty$-contractivity, we will prove a stability estimate with respect to the fluxes.

We will from now on assume that $f(0)=0$ is the minimum of $f$. As in \cite{Caretal2006}, we restrict ourselves to initial data in $\B$, and start by assuming that the support of $u_0$ consists of one connected component. Then, under a certain condition on $f$, we can ensure that the support of the solution to \eqref{eq:cons_law} is connected at any later time $t>0$:
\begin{lemma}\label{lem:connected}
Assume that $f$ is convex and that
\begin{align}\label{eq:f_condition}
\left\| a \right\|_{L^\infty([0,M])} \leq C, \qquad  a(v) := \frac{f'(v)v-f(v)}{v^2},
\end{align}
holds for some constant $C$, where $M=\|u_0\|_{L^\infty(\R)}$. Furthermore, let $u_0 \in \B$ satisfy \eqref{eq:OSLC} and assume that $\supp (u_0)$ consists of one connected component. Then the support of the solution to \eqref{eq:cons_law}, $u(t)$, is connected for any $t>0$.
\end{lemma}
\begin{proof}
Consider the transport equation
\begin{align}\label{eq:transport}
\partial_t v + \partial_x\left(b v \right) = 0, \quad u(0)=u_0, \qquad b(x,t) = \frac{f\left(u(x,t)\right)}{u(x,t)},
\end{align}
where $u(t)$ is the entropy solution to \eqref{eq:cons_law}, with its associated characteristics equation
\begin{align}\label{eq:genchar}
\frac{dX(t)}{dt}=b(x,t), \quad X(0)=x. 
\end{align}
If $b(x,t)$ is $\Lip^+$ bounded with respect to $x$ for all $t>0$, the generalized characteristics of \label{eq:genchar} are unique. It follows that the transport equation \eqref{eq:transport} has a unique measure solution $v(t)=X(t)\# u_0$, the pushforward of $u_0$ by the map $X(t)$, see Poupaud and Rascle \cite{Poup97}. Furthermore, the map $X(x,t)$ is continuous on $\R_+ \times \R$. Thus, if the support of $u_0$ is connected, the support of $v(t) = X(t)\# u_0$ has to be connected as well. 

As $u_0$ is $\Lip^+$ bounded and $f$ is convex, $u(t)$ is also $\Lip^+$ bounded. Then, under the condition \eqref{eq:f_condition}, one can check that $b(t)$ is indeed $\Lip^+$ bounded. (Note that $a \geq 0$ in \eqref{eq:f_condition} as $f(0)=0$ is the minimum of $f$.) It follows that there is a unique measure solution $v(t)=X(t) \# u_0$ to \eqref{eq:transport} with connected support. 

The entropy solution $u(t)$ of \eqref{eq:cons_law} also solves the transport equation \eqref{eq:transport}. Hence, as the measure solution of \eqref{eq:transport} is unique, the entropy solution has to satisfy $u(t)= X(t)_\# u_0$. Thus, the support of $u(t)$ has to be connected. 
\end{proof}
\noindent Note that the above condition \eqref{eq:f_condition} holds for $f$ twice continuously differentiable.

The above lemma makes it possible to find an expression for the inverse of the primitive of $u(t)$ such that we can utilize the interpretation \eqref{eq:winfty_onedim} of $W_\infty$. In the case of an uniformly convex flux function, Carrillo et al. \cite{Caretal2006} make use of the Hopf--Lax formula \eqref{eq:laxhopf} to explicitly express the primitive $U$ of $u$ and then find the inverse. We will do the same, but under the slightly different assumption that $f$ satisfies \eqref{eq:f_condition} and is convex and superlinear to include fluxes of the form \eqref{eq:ftflux}. The resulting explicit expression for the inverse is stated in the lemma below.


\begin{lemma}\label{lem:priminv}
Let $u_0 \in \B$ be such that $\supp(u_0)$ consists of one connected component and let $f$ satisfy \eqref{eq:f_condition}. Then the inverse of $U(t) =\int^x u(t)$, where $u(t)$ solves \eqref{eq:cons_law} is
\begin{align*}
U^{-1}(\gamma,t) = \max_{0\leq \omega \leq \gamma} \left\{t \tilde{f}\left(\frac{\gamma-\omega}{t} \right) + U_0^{-1}(\omega)\right\},
\end{align*}
where $\tilde{f}$ is the inverse of $f^*$, see \eqref{eq:legendre}, restricted to $[0,\infty)$.
\end{lemma}
\begin{proof}
As $u_0 \in \B$, $u(t) \in \B$, and it follows that $U(t)= \int^x u(t) $ is Lipschitz continuous and bounded. Then, as $f$ is convex and superlinear, $U(t)$ can be expressed with the Hopf--Lax formula \eqref{eq:laxhopf}. 

As $\supp (u_0)$ is connected, $U_0$ is strictly increasing from 0 to 1 on a finite interval and we can find its inverse. Furthermore, we know from Lemma \ref{lem:connected} that the support of $u(t)$ is connected. It follows that $U(t)$ is strictly increasing. Hence, its inverse exists and can be implicitly defined by the Hopf--Lax formula.

Note that as $f(0)=0$ is the minimum, $f^*(0)=0$. As $f$ is convex, it is (strictly) increasing on $[0,\infty)$. Then $f^*(p)$ has to be increasing for $p \in [0,\infty)$. The rest of the proof is exactly like the proof of \cite[Lemma 2.3]{Caretal2006}. 
\end{proof}

Next follows a contraction result in the $W_\infty$-distance with respect to the initial data. The result proposed here is \cite[Thm. 2.4, Thm. 2.5]{Caretal2006} adjusted to include the front tracking flux \eqref{eq:ftflux}. As we initially assume that $f$ is only convex, we do not need the approximation procedure of the flux in $C^1$ which is needed in the proofs of \cite[Thm. 2.4, Thm. 2.5]{Caretal2006} to make the contraction estimate valid for convex fluxes. We restate the main details of the proof here for completeness.

\begin{proposition}\label{lem:winfty_init}
Let $u_0,v_0 \in \B\cap BV(\R)$ and let $f$ satisfy \eqref{eq:f_condition} and be convex and superlinear. Then the respective entropy solutions $u(t)$ and $v(t)$ of \eqref{eq:cons_law} satisfy
\begin{align}\label{eq:winfty_init}
W_\infty \big(u(t),v(t)\big) \leq W_\infty (u_0,v_0).
\end{align}
\end{proposition}
\begin{proof}
This proof is very similar to the ones of \cite[Thm. 2.4, Thm. 2.5]{Caretal2006}.

Due to the assumptions on $u_0$ and $f$, the primitive of $u$ can be found by the Hopf--Lax formula \eqref{eq:laxhopf}. We start by assuming that $\supp(u_0)$ consists of one connected component. Then Lemma \ref{lem:priminv} holds, and we can look at the difference between explicit expressions of the inverses,
\begin{align}\label{eq:explicit}
 U^{-1}(\gamma,t)-V^{-1}(\gamma,t)
= \max_{0\leq \omega \leq \gamma} \left\{t \tilde{f}\left(\frac{\gamma-\omega}{t} \right) + U_0^{-1}(\omega)\right\} - \max_{0\leq \omega \leq \gamma} \left\{t \tilde{f}\left(\frac{\gamma-\omega}{t} \right) + V_0^{-1}(\omega)\right\}.
\end{align}
Assume that $\omega_m$ realizes the maximum in the first expression. Then
\begin{align*}
 U^{-1}(\gamma,t)-V^{-1}(\gamma,t) & = t \tilde{f}\left(\frac{\gamma-\omega_m}{t} \right) + U_0^{-1}(\omega_m) - \max_{0\leq \omega \leq \gamma} \left\{t \tilde{f}\left(\frac{\gamma-\omega}{t} \right) + V_0^{-1}(\omega)\right\} \\
& \leq t \tilde{f}\left(\frac{\gamma-\omega_m}{t} \right) + U_0^{-1}(\omega_m) - t \tilde{f}\left(\frac{\gamma-\omega_m}{t} \right) - V_0^{-1}(\omega_m) \\
& = U_0^{-1}(\omega_m) - V_0^{-1}(\omega_m) \leq \sup_{\omega \in [0,1]} \left| \left(U_0^{-1} - V_0^{-1}\right)(\omega)\right|
\end{align*}
By interchanging the roles of $ U^{-1}$ and $ V^{-1}$, we find that
\begin{align*}
\left|U^{-1}(\gamma,t)-V^{-1}(\gamma,t)\right| \leq \sup_{\omega \in [0,1]} \left| \left(U_0^{-1} - V_0^{-1}\right)(\omega)\right|.
\end{align*}
Taking the supremum on the left hand side yields \eqref{eq:winfty_init} for initial data in $\B$ with support consisting of one connected component. 

We extend the result to general initial data in $\B \cap BV (\R)$. Consider two sequences $u_0^n, v_0^n \in \B$ with $\supp(u_0^n)$ and $\supp(v_0^n)$ connected, such that $u_0^n \to u_0$ and $v_0^n \to v_0$  in $L^1(\R)$, and $\|u_0^n \|_{L^\infty}, \|v_0^n \|_{L^\infty} \leq \max \{ \|u_0 \|_{L^\infty}, \|v_0 \|_{L^\infty} \}$. Then, as proven in \cite[Th. 5.5]{Caretal2006}, for any $\epsilon$, we can choose sequences $u_0^n, v_0^n \in \B$ with connected support such that
\begin{align*}
W_\infty \big(u^n_0,v^n_0\big) \leq W_\infty (u_0,v_0) +\epsilon,
\end{align*} 
and as $u_0^n, v_0^n$ have connected supports, we know that
\begin{align}\label{eq:wp_init}
W_p \big(u^n(t),v^n(t)\big) \leq W_\infty \big(u^n(t),v^n(t)\big) \leq W_\infty (u_0,v_0) +\epsilon.
\end{align}
It is well-known that scalar conservation laws satisfy an $L^1$-contraction property for any $t>0$,
\begin{align*}
\|u^n(t)-u(t) \|_{L^1(\R)} &\leq \|u_0^n-u_0 \|_{L^1(\R)}, \\
\|v^n(t)-v(t) \|_{L^1(\R)} &\leq \|v_0^n-v_0 \|_{L^1(\R)}.
\end{align*}
Hence, for any $t \geq 0$, $u^n(t) \to u(t)$, $v^n(t) \to v(t)$ in $L^1(\R)$ as $n \to \infty$. Furthermore, $\|u^n(t)\|_{L^\infty} \leq \|u_0^n \|_{L^\infty} \leq \|u_0 \|_{L^\infty}$, and similarly for $v^n(t)$. It follows that $\supp \left( u^n(t) \right)$ and $\supp \left( v^n(t) \right)$ are uniformly bounded in $n$. Due to the bounded supports, the $p$th order moments of both $u^n(t)$ and $v^n(t)$ will also converge. As convergence in $W_p$ is equivalent to weak convergence and convergence of the $p$th order moment \cite[Thm. 7.12]{Vil03}, we can now take the limit as $n \to \infty$ to the left in \eqref{eq:wp_init},
\begin{align*}
W_p \big(u(t),v(t)\big) \leq W_\infty (u_0,v_0) +\epsilon.
\end{align*}
We conclude the proof by letting $p \to \infty$ and, as the left hand side does not depend on $\epsilon$, we send $\epsilon$ to zero.
\end{proof}
We now turn to the stability estimate in $W_\infty$ with respect to the flux functions.
\begin{proposition}\label{lem:winfty_flux}
Let $u_0 \in \B$, and let $f$ and $g$ satisfy \eqref{eq:f_condition} and be convex and superlinear. Then the respective entropy solutions $u$ and $v$ satisfy
\begin{align}\label{eq:winfty_flux}
W_\infty\big(u(t),v(t)\big) \leq t \sup_{\gamma \in [0,1]}\left|\left(\tilde{f}- \tilde{g}\right)\left(\frac{\gamma}{t}\right)\right|,
\end{align}
where $\tilde{f},\tilde{g}$ are the inverses of the Legendre transforms $f^*$ and $g^*$, restricted to $[0,\infty)$, of $f$ and $g$.
\end{proposition}
\begin{proof}
As in Proposition \ref{lem:winfty_init}, the primitive of $u$ can be found by the Hopf--Lax formula \eqref{eq:laxhopf}, and we start by assuming that $\supp(u_0)$ consists of one connected component such that we have an explicit expression for the difference between the inverses \eqref{eq:explicit}, due to Lemma \ref{lem:priminv}. Again, assume that $\omega_m$ realizes the maximum in the first expression. Then
\begin{align*}
 U^{-1}(\gamma,t)-V^{-1}(\gamma,t) & = t \tilde{f}\left(\frac{\gamma-\omega_m}{t} \right) + U_0^{-1}(\omega_m) - \max_{0\leq \omega \leq \gamma} \left\{t \tilde{g}\left(\frac{\gamma-\omega}{t} \right) + U_0^{-1}(\omega)\right\} \\
& \leq t \tilde{f}\left(\frac{\gamma-\omega_m}{t} \right) - t \tilde{g}\left(\frac{\gamma-\omega_m}{t} \right) \leq t \sup_{0 \leq \omega \leq \gamma } \left| \tilde{f}\left(\frac{\gamma-\omega_m}{t} \right) - \tilde{g}\left(\frac{\gamma-\omega_m}{t} \right)\right|,
\end{align*}
and, after interchanging the roles of $ U^{-1}(\gamma,t)$ and $V^{-1}(\gamma,t)$, we get
\begin{align*}
\left|  U^{-1}(\gamma,t)-V^{-1}(\gamma,t) \right| \leq t \sup_{0 \leq \omega \leq \gamma } \left| \tilde{f}\left(\frac{\gamma-\omega}{t} \right) - \tilde{g}\left(\frac{\gamma-\omega}{t} \right)\right|.
\end{align*}
Taking the supremum over $\gamma \in [0,1]$ results in \eqref{eq:winfty_flux} for $u_0$ with support consisting of one connected component.

We extend the result to general initial data in $\B$. Again we consider a sequence $u_0^n \in \B$ with $\supp(u_0^n)$ connected, such that $u_0^n \to u_0$ in $L^1(\R)$ and $\|u_0^n \|_{L^\infty} \leq \|u_0 \|_{L^\infty}$. Then for $u^n(t),v^n(t)$,
\begin{align}\label{eq:wp_flux}
W_p\big(u^n(t),v^n(t)\big) \leq W_\infty\big(u^n(t),v^n(t)\big) \leq t \sup_{\gamma \in [0,1]}\left|\left(\tilde{f}- \tilde{g}\right)\left(\frac{\gamma}{t}\right)\right|. 
\end{align} 
We use the $L^1$-contraction property,
\begin{align*}
\|u^n(t)-u(t) \|_{L^1(\R)} &\leq \|u_0^n-u_0 \|_{L^1(\R)}, \\
\|v^n(t)-v(t) \|_{L^1(\R)} &\leq \|u_0^n-u_0 \|_{L^1(\R)},
\end{align*}
to conclude that $u^n(t) \to u(t)$, $v^n(t) \to v(t)$ in $L^1(\R)$, and that the $p$th order moments of both $u^n(t)$ and $v^n(t)$ converge. Then we can take the limit as $n \to \infty$ in \eqref{eq:wp_flux},
\begin{align*}
W_p\big(u(t),v(t)\big) \leq t \sup_{\gamma \in [0,1]}\left|\left(\tilde{f}- \tilde{g}\right)\left(\frac{\gamma}{t}\right)\right|.
\end{align*} 
Letting $p \to \infty$ then yields \eqref{eq:winfty_flux}.
\end{proof}
The inequality \eqref{eq:winfty_flux} does not provide much information in general. But, if we consider a convex function $f$ and its piecewise linear interpolation, we show that the right hand side of \eqref{eq:winfty_flux} can be made small in the upcoming lemma. Recall the definition of the Legendre transform in \eqref{eq:legendre}. One can check that the Legendre transform of a piecewise linear, convex and continuous function,
\begin{align}\label{eq:interpol}
g(u) = g_j + \sigma_j(u-u_j), \quad \sigma_ j = \frac{g_{j+1}-g_j}{u_{j+1}-u_j}, \quad u_j \leq u \leq u_{j+1}, \quad j \in \Z\cap[-J,J],
\end{align}
is
\begin{align}\label{eq:lin_legendre}
g^*(p)= 
\begin{cases}
+ \infty, \qquad & p > \sigma_{-J}, \quad p< \sigma_J, \\
u_j p -g_j, \qquad & \sigma_{j-1} \leq p \leq \sigma_j.
\end{cases}
\end{align}

\begin{lemma}\label{lem:fg_zeroes}
Let $f$ be convex and superlinear, and let $g$ be the piecewise linear interpolation \eqref{eq:interpol} of $f$ with $g_j=f(u_j)$. Then
\begin{align*}
f^*(p)-g^*(p) & = 0, \qquad \textrm{for some} \quad p \in [\sigma_{j-1},\sigma_{j}], \ \quad \textrm{for all} \quad j \in \mathbb{Z}.
\end{align*}
\end{lemma}
\begin{proof}
As $f$ is convex and superlinear $g$ will also be. The same is true for $f^*$ and $g^*$. Also notice that $f\leq g$, so that $f^* \geq g^*$. 

From \eqref{eq:lin_legendre} we see that $g^*(p)= u_jp-f(u_j)$ for $p \in [\sigma_{j-1},\sigma_{j}]$ so that
\begin{align}\label{eq:expression}
f^*(p)-g^*(p) = \sup_{u \in \R} \left\{pu-f(u)\right\} - \left( u_jp-f(u_j) \right), \quad p \in [\sigma_{j-1},\sigma_{j}].
\end{align}
We claim that there exists $p \in [\sigma_{j-1},\sigma_{j}]$ such that $u_j$ realizes the supremum in the expression for $f^*$. For this $p$ \eqref{eq:expression} will be zero. For $p \in \partial f(u_j)$ (the sub-differential at $u_j$), $u_j$ realizes the supremum in \eqref{eq:expression}. If $\partial f(u_j) \subset [\sigma_{j-1},\sigma_{j}]$ we're done. But this has to be true as $f\leq g$, $f(u_j)-g(u_j)=0$ and $f$ is convex. 
\end{proof}

Fully equipped with estimates in $W_\infty$, we prove the $\Dx$ convergence rate.
\begin{theorem}\label{thrm:wass_infty}
Let $u_0 \in \B \cap BV(\R)$ and let $f$ be twice continuously differentiable and convex. Then the front tracking approximation converges towards the entropy solution of \eqref{eq:cons_law} at a rate of $O(\Dx)$, i.e.
\begin{align*}
W_\infty (u(t), u^{\delta,\Dx}(t)) \leq L(t) \Dx,  
\end{align*}
where $L(t)=1+t\max_{u \in [0,M +\delta]}|f''(u)|$.
\end{theorem} 
\begin{proof}
By the triangle inequality we have
\begin{align*}
W_\infty (u(t), u^{\delta,\Dx}(t)) \leq W_\infty (u(t), u^{\Dx}(t)) + W_\infty (u^\Dx(t), u^{\delta,\Dx}(t)).
\end{align*}
As $f$ is $C^2$, one can check that the condition \eqref{eq:f_condition} holds for both $f$ and the approximation $f^\delta$ in \eqref{eq:ftflux}. Then, by Proposition \ref{lem:winfty_init},
\begin{align*}
W_\infty (u(t), u^{\Dx}(t)) \leq W_\infty (u_0, u^{\Dx}_0) \leq \Dx,
\end{align*}
where the last inequality follows from the fact that the primitives of $u_0$ and \eqref{eq:ftinitial} are both increasing and satisfy $U_0((i+\hf)\Dx)-U_0^\Dx((i+\hf)\Dx)=0$ for all $i \in \mathbb{Z}$. 

From Proposition \ref{lem:winfty_flux}, we have
\begin{align*}
W_\infty (u^\Dx(t), u^{\delta,\Dx}(t)) & \leq t \sup_{\gamma \in [0,1]}\left|\left(\tilde{f}- \tilde{g}\right)\left(\frac{\gamma}{t}\right)\right| \\
& = t \sup_{\gamma \in [0,1]}\left|\left((f^*)^{-1}- (g^*)^{-1}\right)\left(\frac{\gamma}{t}\right)\right| \\
& \leq t \max_i | \sigma_{j+1} - \sigma_{j-1} |= t \max_{u \in [0,M+\delta]}|f''(u)|2 \delta,
\end{align*}
where we in the last step have used Lemma \ref{lem:fg_zeroes} and the fact that $f^*$ and $g^*$ are increasing for $\gamma/t>0$.
\end{proof}
Finally, having established Theorem \ref{thrm:wass_infty}, we can conclude that the main theorem of this paper, Theorem \ref{thrm:main}, holds.

\section{Concluding remarks}\label{sec:conclude}
In this paper we have shown that the front tracking approximations to scalar one-dimensional conservation laws with convex fluxes converge at a rate of $\Dx^{1+\unitfrac{1}{p}}$ in the $p$-Wasserstein distance. This gives the front tracking method an advantage, in terms of guaranteed convergence rate, over (formally) second-order finite volume schemes for which no second-order convergence rate has been proven for general initial data. 

The convergence rate results in this paper are limited to $\Lip^+$ bounded initial data $u_0$. In the case of $\Lip^+$ unbounded $u_0$, it is well-known that the solution to \eqref{eq:cons_law} satisfies
\begin{align}\label{eq:strongOSLC}
\frac{u(x+z,t)-u(x,t)}{z} \leq \frac{C}{t}, \quad t>0,
\end{align}
whenever $f$ is strongly convex, $f'' \geq \alpha > 0$. Let $u(t)$ and $v(t)$ be solutions to \eqref{eq:2cons_laws}, where $v_0$ is the piecewise constant projection \eqref{eq:ftinitial} to $u_0$, and $\|g-f\|_{L^\infty} \leq O(\Dx^2)$, where $g$ is strongly convex as well. Then (by an approach similar to the one in Proposition \ref{lem:w1stability}) preliminary calculations indicate that $ W_1\big(u(t),v(t)\big) = O\left( \Dx^2 \log |\Dx|\right)$ for $\Lip^+$ unbounded $u_0$. The front tracking flux $g=f^\delta$ is piecewise linear, but as it is an approximation to the strongly convex function $f$, the front tracking approximation should satisfy a discrete version of \eqref{eq:strongOSLC}. This might be sufficient to prove a convergence rate of $\Dx^2 \log |\Dx|$ in the $\Lip^+$ unbounded case, but it needs to be investigated further.

The main theorem in this paper strongly depend on the convexity of the flux $f$. As mentioned in Remark \ref{rem:hjft}, the $\Dx^2$ convergence rate in $W_1$ that one can deduce from \cite{KR02} can be extended to non-convex fluxes as long as $u_0 \in C_c^1$. Remark \ref{rem:hjft} and the discussion on $\Lip^+$ unbounded $u_0$ indicate that the rate might be lower for more general initial data. Whether the $\Dx$ rate in $W_\infty$ can be extended to the non-convex case is unclear. The proofs in Section \ref{sec:winfty} depend on an explicit expression for the generalized inverse of the primitive. Due to the more complex nature of $u(t)$ in the non-convex case, a feasible expression for the generalized inverse is currently out of reach.

%
%
%

\end{document}